\documentclass{article}

\usepackage{amsmath}
\usepackage{amssymb}
\usepackage{amsthm}
\usepackage[german, english]{babel}
\usepackage[utf8]{inputenc}
\usepackage{scrextend}
\usepackage{xcolor}
\usepackage[normalem]{ulem}
\usepackage[arrow, matrix, curve]{xy}
\usepackage{hyperref}

\def\inv{^{-1}}

\def\Z{{\mathbb Z}}

\newcommand{\oeis}[1]{\text{\href{https://oeis.org/#1}{{\small \tt #1}}}}

\newtheorem{thm}{Theorem}

\usepackage{geometry}
\usepackage{graphicx}
\usepackage{algpseudocode}
\usepackage{paralist}
\usepackage{xcolor}
\def\red{\color{red}}

\newtheorem{cor}[thm]{Corollary}
\newtheorem{defi}[thm]{Definition}

\begin{document}

\title{The vectorial kernel method for walks with longer steps}
\author{Valerie Roitner\thanks{TU Wien, Institute for discrete mathematics and geometry. ORCiD of the author: \href{https://orcid.org/0000-0002-2621-431X}{0000-0002-2621-431X}}}

\maketitle

\begin{abstract}
Asinowski, Bacher, Banderier and Gittenberger \cite{ABBG-vkm} recently developed the vectorial kernel method -- a powerful extension of the classical kernel method that can be used for paths that obey constraints that can be described by finite automata, e.g. avoid a fixed pattern, avoid several patterns at once, stay in a horizontal strip and many others more. However, they only considered walks with steps of length one. In this paper we will generalize their results to walks with longer steps. We will also give some applications of this extension and prove a conjecture about the asymptotic behavior of the expected number of ascents in Schröder paths.
\end{abstract}

\section{Introduction}

Lattice path structures appear often in mathematical models in natural sciences or computer science, for example in analysis of algorithms (see e.g. \cite{Knuth3, ADHKP}) or physics when modeling wetting and melting processes \cite{Fisher} or Brownian motion \cite{Mar-brownian}. Another field is bioinformatics (\cite{MTM-dna, RG-dna, JQR}), where the lattice paths are usually constrained to avoid certain patterns.

The generating function of these objects can often be described by a functional equation that can be solved by the kernel method.

In its easiest form, i.e., for solving equations of the type
$$K(z,u)F(z,u)=A(z,u)+B(z,u)G(z),$$
where $K,A$ and $B$ are known functions and $F$ and $G$ are unknown functions, where $K(z,u)=0$ has only one small root (i.e. a root $u_i(z)$ with $u_i(z)\sim 0$ as $z\sim 0$), the kernel method has been folklore in combinatorics and related fields like probability theory. One identifiable source is Knuth's book \cite{Knuth} from 1968, where he used this idea as a new method for solving the ballot problem. Ever since there have been several extensions and applications of this method, see for example \cite{BW, BMJ, BMM}, one of the most recent being the \em{}vectorial kernel method\em{} by Asinowski, Bacher, Banderier and Gittenberger \cite{ABBG-vkm}. It allows to solve enumeration problems for lattice paths obeying constraints that can be described by a finite automaton. Furthermore, it also allows enumeration of the occurrence of any phenomenon that can be described by a finite automaton, e.g. the number of occurrences of a given pattern.

In their paper they only considered directed walks with steps where the first coordinate is 1. However, this method can be generalized to directed walks with longer steps which will be done in this paper. The proofs used here follow the same methods as in the case with steps of length one, but some adaptions have to be made.

In the final two chapters we will have a look at some applications of this method. Firstly, we will re-derive the number of Schröder paths (excursions with steps $U=(1,1), D=(1,-1)$ and $F=(2,0)$) avoiding the pattern $UF$, which has been studied in \cite{Yan}. Furthermore, we will derive the trivariate generating function for the number of ascents (i.e. number of sequences of nonempty consecutive up-steps) in Schröder paths and prove that the expected number of ascents in Schröder paths of length $2n$ indeed behaves asymptotically like $(\sqrt{2}-1)n$ as Callan conjectured on the OEIS \cite{OEIS}, entry \oeis{A090981}.

\section{Definitions and notations}

A \em{}lattice path\em{} in $\mathbb{Z}^2$ is a finite sequence or finite word $w=[\nu_1,\dots,\nu_m]$ such that all vectors $\nu_i$ lie in the \em{}step set\em{} $\mathcal{S}$, which is a finite subset of $\mathbb{Z}^2.$ A lattice path can be visualized as a polygonal line in the plane, which is created by starting at the origin and successively appending the vectors $\nu_i=(u_i,v_i)$ at the end. The vectors $\nu_i$ are called \em{}steps\em{}. In this paper we will only consider \em{}directed\em{} lattice paths where all steps $(u_i,v_i)$ have a positive first entry.

\begin{figure}[h]
\begin{center}
\includegraphics[height=3cm]{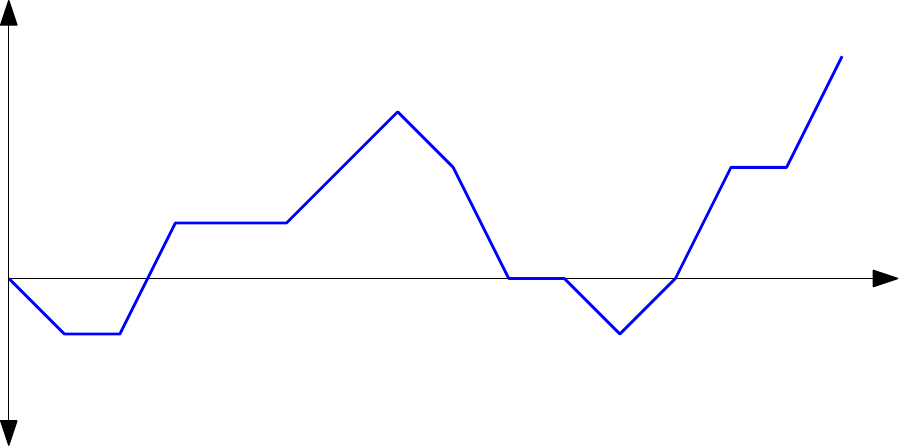}
\caption{A lattice path} \label{ex-lp}
\end{center}
\end{figure}

The first entry $u_i$ of a step is called its \em{}length\em{} and the length of a walk $w$, denoted by $|w|$ is the sum of the length of all its steps, i.e. $|w|=u_1+\dots+u_m$. This does not always coincide with the number of steps, only if $u_i=1$ for all steps in $\mathcal{S}$ this is the case. The \em{}final altitude\em{} of a walk $w$, denoted by $\text{alt}(w)$ is the sum of the altitudes of all steps, i.e. $\text{alt}(w)=v_1+\dots+v_m$. Thus, a walk starting in $(0,0)$ terminates in $(|w|,\text{alt}(w))$.

The \em{}step polynomial\em{} $P(t,u)$ of the step set $\mathcal{S}$ is given by
\begin{equation}
\label{steppoly}
P(t,u)=\sum_{s\in\mathcal{S}} t^{|s|}u^{\text{alt}(s)}.
\end{equation}
The variable $t$ encodes length, the variable $u$ encodes altitude. When all steps have length one, we can omit the dependency on $t$ and write
$$P(u)=\sum_{s\in\mathcal{S}}u^{\text{alt}(s)}.$$
Denote $-c$ the smallest (negative) power of $u$ in $P(t,u)$ and $d$ the largest (positive) power of $u$.  If $\mathcal{S}$ only contains negative or only positive altitudes of steps, the following results still hold, but the corresponding models are easy to solve and lead to rational generating functions.

Often, there are constraints imposed on the lattice paths one wants to consider, e.g., the path is not allowed to leave a certain region or has to end at a certain altitude, usually at altitude zero. This leads to the following

\begin{defi} For lattice paths obeying constraints we define:
\begin{itemize}
\item
A \em{}walk\em{} is an unconstrained lattice path.
\item
A \em{}bridge\em{} is a lattice path whose endpoint lies on the $x$-axis.
\item
A \em{}meander\em{} is a lattice path that lies in the quarter-plane $\Z_{\geq0}\times\Z_{\geq0}$. Since we only consider lattice paths with steps with positive $x$-coordinates, this is equivalent to lattice paths that never attain negative altitude.
\item
An \em{}excursion\em{} is a lattice path that is both a bridge and a meander, i.e., a lattice path that ends on the $x$-axis, but never crosses the $x$-axis.
\end{itemize}
\end{defi}

Banderier and Flajolet \cite{BF} computed generating functions for all these classes of lattice paths. Their results can be summarized by the table in figure \ref{table-BF}:\\
\begin{figure}[ht]
\begin{center}
\includegraphics[height=8cm]{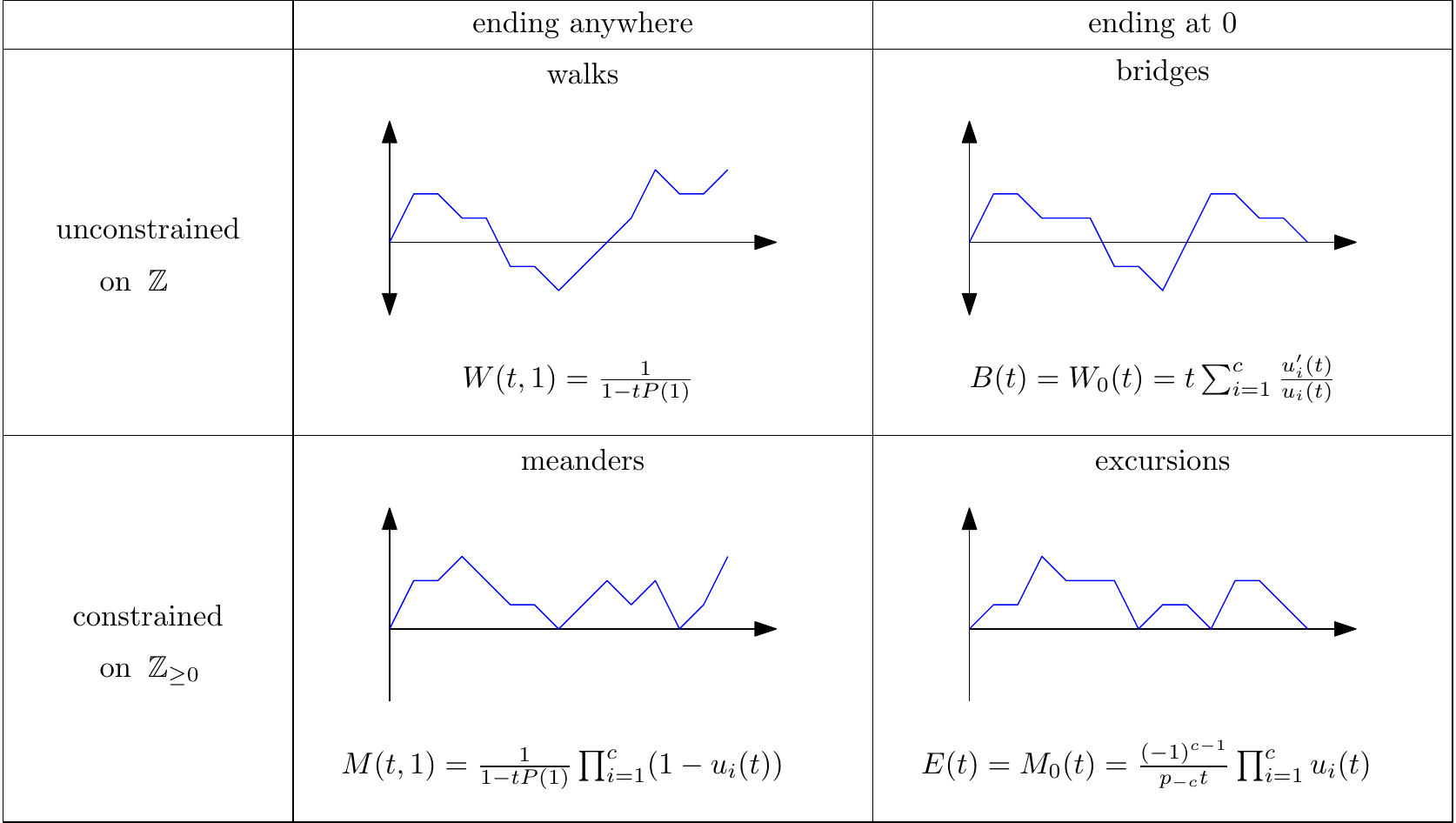}
\caption{The generating functions for walks, bridges, meanders and excursions (in the case of steps of length one). Here, $P(u)$ is the step polynomial, $c$ is the number of small roots, which are given by $u_1,\dots, u_c$. Here, $W_0(t)$ stands for $[u^0]W(t,u)=W(t,0)$ (and analogously for $M_0$).} \label{table-BF}.
\end{center}
\end{figure}
\\
This study was generalized in \cite{ABBG-vkm} to paths with steps of length one that avoid one single pattern. In this paper we will show similar results for walks with longer steps.\\
\\
A \em{}pattern\em{} is a fixed path
$$p=[a_1,\dots,a_{\ell}]$$
where $a_i\in \mathcal{S}$. The length of a pattern is the sum of the lengths of its steps. An  \em{}occurrence\em{}  of a pattern $p$ in a lattice path $w$ is a contiguous sub-string of $w$, which coincides with $p$. We say a lattice path $w$ \em{}avoids\em{} the pattern $p$ if there is no occurrence of $p$ in $w$. For example, the path $w=[(1,1),(3,0),(3,0),(1,1), (1,-2),(3,0),(1,1)]$ has two occurrences of the pattern $[(3,0),(1,1)]$ but avoids the pattern $[(1,-2),(1,-2)]$.

A \em{}prefix\em{} of length $k$ of a string is a contiguous non-empty sub-string that matches the first $k$ letters (or steps, to phrase it with words more familiar for a lattice path setting). Similarly, a \em{}suffix\em{} of length $k$ of a string is a contiguous non-empty sub-string that matches the last $k$ letters. For example, $[(1,1),(3,0),(3,0)]$ is a prefix (of length 3) of the path from the previous example and $[(1,-2),(3,0),(1,1)]$ is a suffix. A \em{}presuffix\em{} of a pattern is a non-empty string that is both prefix and suffix. In our above example, $[(1,1)]$ is the only presuffix of this given path.

Some authors use a different definition of a pattern, namely when the pattern is contained in the path as non-contiguous substring, see for example \cite{BBLGPW}. The path $w$ as defined in the previous example contains $[(1,1),(3,0),(1,1),(3,0)]$ in the non-contiguous-sense, but not in the contiguous sense. Lattice paths avoiding patterns in the non-contiguous sense also can be dealt with the vectorial kernel method. In this paper we will only consider consecutive patterns.

\vspace{0.5cm}
\noindent
In order to describe pattern avoidance we will need the concept of finite automata.

\begin{defi}
A \em{}finite automaton\em{} is a quadruple $(\Sigma, \mathcal{M}, s_0, \delta)$ where
\begin{itemize}
\item
$\Sigma$ is the input alphabet (in our case, $\Sigma$ will usually be the step set)
\item
$\mathcal{M}$ is a finite, nonempty set of states
\item
$s_0\in \mathcal{M}$ is the initial state
\item
$\delta: \mathcal{M}\times\Sigma\to \mathcal{M}$ is the state transition function. In many cases, it is useful to allow $\delta$ to be a partial function as well, i.e., not every input $\delta(S_i,x)$ has to be defined. Especially for pattern avoidance the usage of partial functions is very helpful.
\end{itemize}
Sometimes there is also a set $F\subseteq\mathcal{M}$ of final states given in the definition of a finite automaton. Here, however, we will not have any final states (i.e. $F=\emptyset$).
\end{defi}

A finite automaton can be described as a weighted directed graph (the states being the vertices, the edges and their weights given by the transition function) or by an adjacency matrix $A$, where the entry $A_{ij}$ consists of the sum of all letters $x$ that, when being in state $S_i$ and reading the letter $x$, transition to state $S_j$. Phrased differently, 
$${A_{ij}=\sum_{x:\delta(S_i,x)=S_j} x}.$$

\vspace{0.5cm}
\noindent
\textbf{Example:} Let $\mathcal{S}=\{U,F,D\}$ where $U=(1,1), F=(2,0)$ and $D=(1,-1)$ be the step set and $p=[U,F,U,D]$ the forbidden pattern. We will build an automaton with $s=4$ states, where $s$ is the number of steps in the pattern. Each state corresponds to a proper prefix of $p$ collected so far by walking along the lattice path. Let us label these states $X_0,\dots, X_{s-1}$ (in our case $X_0,\dots, X_3$). The first state $X_0$ is labeled by the empty word. The next states are labeled by proper prefixes of $p$, more precisely $X_i$ is labeled by $X_i=[a_1,\dots,a_i]$ where $a_j$ are the letters of the forbidden pattern. For $i,j\in\{1,\dots,s\}$ we have $\delta(X_i,\lambda)=X_j$ (or, in the graph setting, an arrow labeled $\lambda$) if $j$ is the maximal number such that $X_j$ is a suffix of $X_i\lambda$.

When the automaton reads a path $w$, it ends in the state labeled with the longest prefix of $p$ that coincides with a suffix of $w$. The automaton is completely determined by the step set and the pattern.
\begin{figure}[h]
\begin{center}
\includegraphics[height=3.5cm]{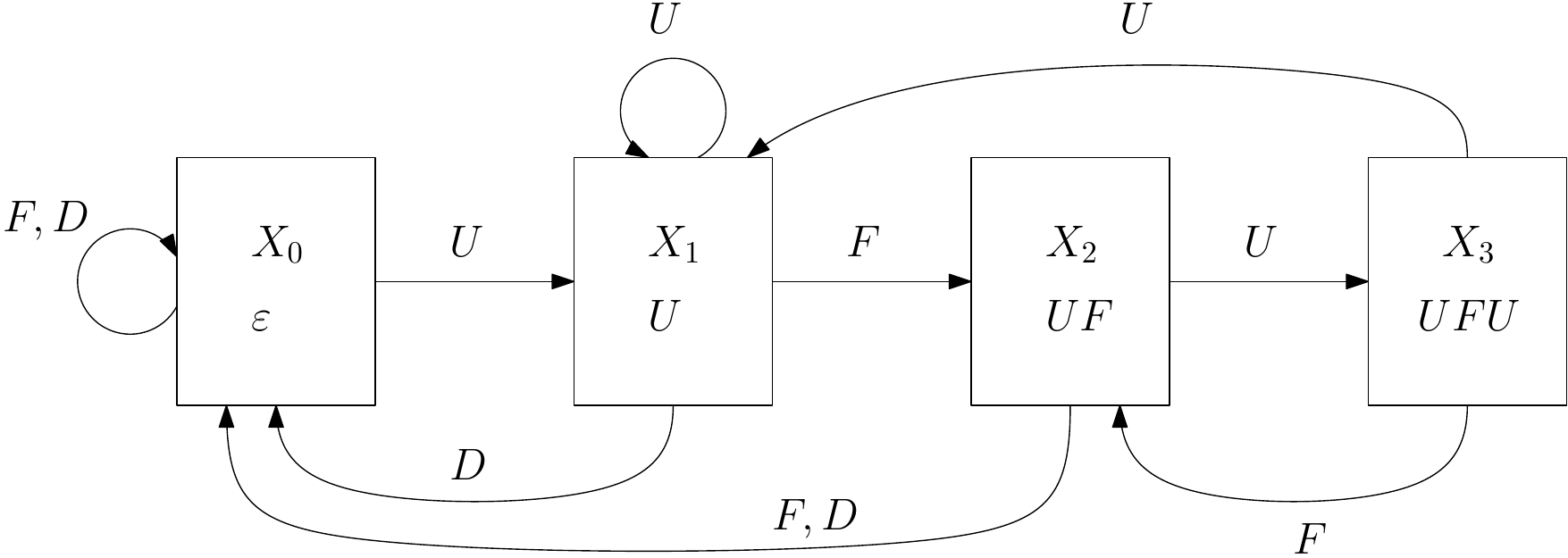}
\caption{The automaton for $\mathcal{S}=\{U,F,D\}$ and $p=[U,F,U,D]$} \label{ex-UFUD}
\end{center}
\end{figure}

When looking at the adjacency matrix of this automaton we also have to keep track of the length of the steps. We obtain
$$A=\begin{pmatrix}
t^2+tu^{-1} & tu & 0 & 0\\
tu\inv & tu & t^2 & 0\\
t^2+tu\inv & 0 & 0 & tu\\
0& tu & t^2 & 0
\end{pmatrix}.$$
In each row except the last one, all entries sum up to $P(t,u)$, because at each state except the last one, all possible steps are allowed. The entries in the last row of the matrix sum up to $P(u)-w_s$, where $w_s$ is the weight of the last step in the forbidden pattern $p$. This is because in the last state $X_{s-1}$ all steps except the one that would make $p$ complete.

\vspace{0.5cm}
\noindent
Automata can not only be used to describe the avoidance of one pattern, but also for other constraints, e.g. the avoidance of several patterns at once (see \cite{ABR}) or height constraints. Or to describe the avoidance of patterns in the non-contiguous sense.

\begin{defi}
The \em{}kernel\em{} of an automaton is defined to be the determinant of $I-A(t,u)$, where $A$ is the adjacency matrix of the automaton, i.e.,
$$K(t,u):=\det(I-A(t,u)).$$
\end{defi}

For certain kinds of automata, for example the automata that arise when considering walks that avoid a pattern, there are easier expressions for the kernel that avoid the computation of the adjacency matrix and its determinant. For more details on this, see \cite{ABBG-vkm}.

\section{The vectorial kernel method for walks with longer steps}

The vectorial kernel method indeed works for walks with longer steps if the right adaptions are made. Instead of the adjacency matrix $A=A(u)$ we now have to consider the adjacency matrix $A(t,u)$ that takes into account the different lengths of the steps by weighting them with the corresponding powers of $t$, i.e. a step of length $i$ is weighted with $t^{i}$. With these adapted adjacency matrix we obtain the following theorems:

\begin{thm}
\label{walks-longsteps}
The bivariate generating function for walks obeying constraints that can be described by a finite automaton (e.g. pattern avoidance, height restrictions, etc.) is given by
\begin{equation}
\label{gf-walks-longsteps}
W(t,u)=\frac{(1,0,\dots,0)\mathrm{adj}(I-A(t,u))\vec{\mathbf{1}}}{\det(I-A(t,u))}
\end{equation}
where $t$ encodes length and $u$ encodes final altitude.
\end{thm}

\begin{thm}
\label{meanders-longsteps}
The bivariate generating function for meanders obeying constraints that can be described by a finite automaton is given by
\begin{equation}
\label{gf-meanders-longsteps}
M(t,u)=\frac{G(t,u)}{u^{e}K(t,u)}\prod_{i=1}^{e}(u-u_i(t))
\end{equation}
where $t$ encodes length and $u$ encodes final altitude, $u_i$ $(i=1,\dots, e)$ are the small roots of $K(t,u)$ and $G(t,u)$ is a polynomial in $u$ which will be characterized in (\ref{G}).
\end{thm}

\textit{Proof of Theorem \ref{walks-longsteps}:} The proof follows the same idea as in the case with steps of length one, which was considered in \cite{ABBG-vkm}. Writing $W_i:=W_i(t,u)$ for the generating function of walks ending in state $X_i$ and using a step-by-step-construction we obtain the following functional equation
$$(W_1,\dots,W_{\ell})=(1,0,\dots, 0)+(W_1,\dots, W_{\ell})\cdot A(t,u),$$
or equivalently
$$(W_1,\dots,W_{\ell})(I-A(t,u))=(1,0,\dots, 0).$$
Multiplying this from the right with $(I-A(t,u))^{-1}=\frac{\text{adj}(I-A(t,u))}{\det(I-A(t,u))}$ we obtain
$$(W_1,\dots,W_{\ell})=\frac{(1,0,\dots,0)\text{adj}(I-A(t,u))}{\det(I-A(t,u))}.$$
The generating function $W(t,u)$ is the sum of the generating functions $W_i(t,u)$ thus we have that
$$W(t,u)=(W_1,\dots,W_{\ell})\vec{\mathbf{1}}=\frac{(1,0,\dots,0)\text{adj}(I-A(t,u))\vec{\mathbf{1}}}{\det(I-A(t,u))}$$
which finishes the proof. \hfill$\Box$

\begin{cor}
The generating function for bridges is given by
$$B(t)=[u^0]W(t,u)=\frac{1}{2\pi i}\int_{|u|=\varepsilon}\frac{W(t,u)}{u}=\sum_{i=1}^{e}\mathrm{Res}_{u=u_i}\frac{W(t,u)}{u}.$$
\end{cor}

\textit{Proof of Theorem \ref{meanders-longsteps}:} This proof works similarly as the one for walks, only that now we also have to take care of the fact that the walk is not allowed to attain negative altitude. Writing $M_i=M_i(t,u)$ for the generating function of meanders ending in state $X_i$ of the automaton and using a step-by step construction we obtain the following vectorial functional equation
$$(M_1,\dots, M_{\ell})=(1,0,\dots, 0)+ (M_1,\dots, M_{\ell})\cdot A(t,u)-\{u^{<0}\}((M_1,\dots, M_{\ell})\cdot A(t,u)).$$
This is equivalent to
$$(M_1,\dots, M_{\ell})(I-A(t,u))=(1,0,\dots, 0)-\{u^{<0}\}((M_1,\dots, M_{\ell})\cdot A(t,u)).$$
Writing $F:=(F_1,\dots,F_{\ell})$ for the right hand side of the above equation we obtain
\begin{equation}
\label{blue}
(M_1,\dots, M_{\ell})(I-A(t,u))=(F_1,\dots, F_{\ell}).
\end{equation}
Multiplying (\ref{blue}) from the right by $(I-A(t,u))^{-1}=\frac{\text{adj}(I-A(t,u))}{\det(I-A(t,u))}$ we obtain
$$(M_1,\dots, M_{\ell})=(F_1,\dots,F_{\ell})\cdot \frac{\text{adj}(I-A(t,u))}{\det(I-A(t,u))}.$$
The generating function $M(t,u)$ is the sum of all the generating functions $M_i$. Using this, defining
$$\vec{v}:=\text{adj}(I-A(t,u))\vec{\mathbf{1}}$$
and using
$$\det(I-A(t,u))=K(t,u)$$
we obtain
\begin{equation}
\label{green}
M(t,u)=\frac{(F_1,\dots, F_{\ell})\vec{v}}{K(t,u)}.
\end{equation}
Let $u_i=u_i(t)$ be a small root of the kernel $K(t,u)$. We plug $u=u_i$ into (\ref{blue}). The matrix $(I-A(t,u))|_{u=u_i}$ is then singular. Furthermore, we observe that $\vec{v}_{u=u_i}$ is an eigenvector of $(I-A(t,u))|_{u=u_i}$ for the eigenvalue $\lambda=0$.

Thus, multiplying (\ref{blue}) from right with $\vec{v}_{u=u_i}$ the left hand side of the equation vanishes. Said differently, the equation
$$(F_1(t,u),\dots, F_{\ell}(t,u))\vec{v}(t,u)=0$$
is satisfied by all small roots $u_i(t)$ of $K(t,u)$.

Let
\begin{equation}
\label{Phi}
\Phi(t,u):=u^{e}(F_1(t,u),\dots, F_{\ell}(t,u))\vec{v}(t,u).
\end{equation}
Note that $\Phi$ is a Laurent polynomial in $u$, because $F_i$ and $\vec{v}$ are Laurent polynomials in $u$ by construction. Because of (\ref{green}) we have that
$$\Phi(t,u)=u^{e}M(t,u)K(t,u)$$
and since $M$ is a power series in $u$ and $K$ has exactly $e$ small roots the Laurent-polynomial $\Phi$ contains no negative powers in $u$ and is a polynomial in $u$. Each small root $u_i$ is a root of the polynomial equation $\Phi(t,u)=0$, thus we have that
\begin{equation}
\label{G}
\Phi(t,u)=G(t,u)\prod_{i=1}^{e} (u-u_i(t))
\end{equation}
where $G(t,u)$ is a polynomial in $u$ and formal power series in $t$. It can be computed by comparing coefficients. Plugging $G$ in (\ref{green}) we obtain
$$M(t,u)=\frac{G(t,u)}{u^{e}K(t,u)}\prod_{i=1}^{e}(u-u_i(t))$$
which finishes the proof. \hfill$\Box$

\begin{cor}
The generating function $E(t)$ for excursions with restrictions described by a finite automaton $A(t,u)$ satisfies
$$E(t)=M(t,0)=\left.\frac{G(t,u)}{u^{e}K(t,u)}\prod_{i=1}^{e}(u-u_i(t))\right|_{u=0}.$$
\end{cor}

\section{Examples}

In this section we will consider some examples illustrating applications of the previous theorems. The first example is more of the simple and introductory kind and deals with Schröder paths avoiding the pattern $UF$, the second one counts Schröder paths having $k$ ascents and proves a conjecture about the asymptotic behavior of the expected number of ascents.

\subsection{Number of Schröder paths of semilength $n$ avoiding $UF$}

Schröder paths are lattice paths consisting of the steps $U=(1,1), D=(1,-1)$ and $F=(2,0)$ which start at $(0,0)$, end at $(2n,0)$ and never go below the $x$-axis. In this section we are dealing with Schröder paths of length $2n$ avoiding $p=UF$. These objects are enumerated by OEIS \oeis{A007317} and have been studied by Yan in \cite{Yan}, where a bijection with Schröder paths without peaks at even level as well as two pattern avoiding partitions were constructed.

The generating function for Schröder paths avoiding $UF$ can be obtained by a first passage decomposition -- if $S^*$ denotes all Schröder paths avoiding $UF$, then
$$S^*=\varepsilon \cup F\times S^* \cup UD \times S^* \cup U \times (S^*\setminus\{\varepsilon \cup F\times S^*)\times D \times S^*,$$
i.e. a Schröder path avoiding $UF$ is either empty, or starts with either F followed by another Schröder path avoiding $UF$, UD and another Schröder path avoiding $UF$ or starts with an up step, followed by an nonempty Schröder path avoiding $UF$ which does not start with F, a down step to altitude zero (the first passage) and another Schröder path avoiding $UF$. For generating functions, this translates to
$$F(x)=1+2xF(x)+x(F(x)-1-xF(x))F(x),$$
where $x$ encodes semilength. From here, the generating function can easily be obtained by solving a quadratic equation. However, in many cases a first passage decomposition does not work while the enumeration problem can still be solved by the vectorial kernel method.

The automaton describing Schröder paths avoiding $UF$ is given by

\begin{center}
\includegraphics[height=3cm]{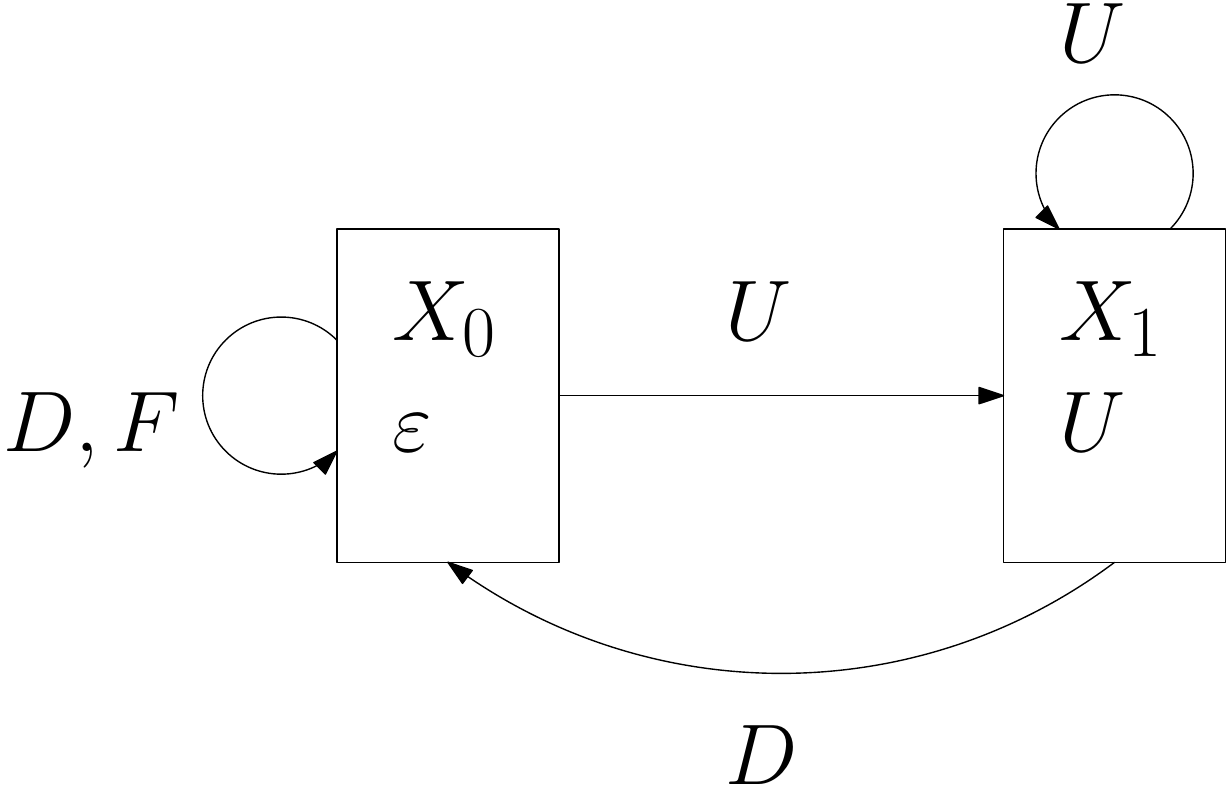}
\end{center}

Its adjacency matrix is
$$A(t,u)=
\begin{pmatrix}
t^2+tu^{-1} & tu\\
tu^{-1} & tu\\
\end{pmatrix}.
$$
Thus the kernel is given by
$$K(t,u)=\det(I-A)=\frac{t^3u^2-t^2u-tu^2-t+u}{u}.$$
Its roots are
$$u_{1/2}=\frac{1-t^2\pm\sqrt{1-6t^2+5t^4}}{2t(1-t^2)},$$
the root with minus being the small root.

Denote $M_0$ the generating function of the walks ending in state $X_0$, i.e., with a $D$ or $F$-step, and $M_1$ the generating function of the walks ending in state $X_1$, i.e., in an $U$-step. Via a step-by-step-construction we obtain the following system of equations for the generating functions:
$$(M_0,M_1)=1+(M_0,M_1)A-\{u^{<0}\}(M_0,M_1)A.$$
This can be rephrased as
$$(M_0,M_1)(I-A)=1-\{u^{<0}\}(M_0,M_1)A.$$
We have that
$$\{u^{<0}\}(M_0,M_1)A=(tu^{-1}m_0,0),$$
where $m_0=[u^0]M_0+M_1$. Thus the forbidden vector $F$ is
$$F=1-\{u^{<0}\}(M_0,M_1)A=(1-tu^{-1}m_0, 0).$$
Using
$$\text{adj}(I-A)=
\begin{pmatrix}
1-tu & tu\\
tu^{-1} & 1-tu^{-1}-t^2
\end{pmatrix}$$
we obtain
$$\vec{v}=\text{adj}(I-A)\cdot\begin{pmatrix} 1\\ 1\end{pmatrix}=\begin{pmatrix} 1\\1-t^2\end{pmatrix}.$$
Thus we have that
$$\Phi(t,u)=u^eF\vec{v}=u-tm_0.$$
Using 
$$\Phi(t,u)=G(t,u)(u-u_1)$$
and comparing coefficients we obtain
$$G(t,u)=1.$$
Using
$$M(t,u)=\frac{G(t,u)}{u^e K(t,u)}(u-u_1(t))=\frac{1}{t^3u-t^2u-tu^2-t+u}\left(u-\frac{1-t^2-\sqrt{1-6t^2+5t^4}}{2t(1-t^2)}\right)$$
we obtain for the generating function $M(t)$ of meanders
$$M(t)=M(t,1)=\frac {2\,{t}^{3}-{t}^{2}-2\,t-\sqrt {5\,{t}^{4}-6\,{t}^{2}+1}+1}{2t \left( {t}^{2}-1 \right)  \left( {t}^{3}-{t}^{2}-2\,t+1 \right) }$$
and the generating function $E(t)$ of excursions
$$E(t)=M(t,0)=\frac{1-t^2-\sqrt{1-6t^2+5t^4}}{2t^2(1-t^2)}.$$
Making a transition to semilength (i.e., the substitution $x:=t^2$) we obtain exactly the same result for the generating function as in \cite{Yan}.

\subsection{Schröder paths of semilength n having $k$ ascents}

\begin{defi}
An \em{ascent}\em{} in a Schröder path is a maximal string of up-steps.
\end{defi}

\begin{figure}[htb!]
\begin{center}
\includegraphics[scale=0.6]{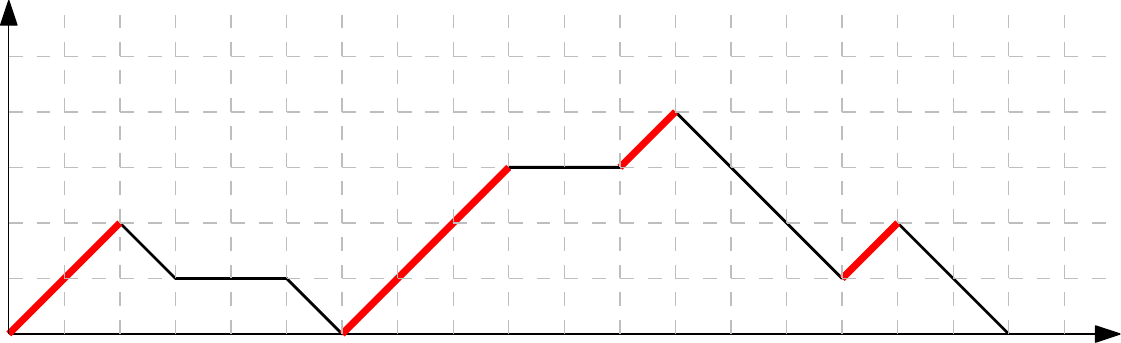}
\caption{A Schröder path with $k=4$ ascents (marked in red).} \label{pic-asc}
\end{center}
\end{figure}

\begin{thm}
\label{schroeder-asc}
Let $X_n$ be the random variable counting ascents in a Schröder path of length $2n$ which is chosen uniformly at random among all Schröder paths of length $2n$. Then $\mathbb{E} (X_n) \sim (\sqrt{2} -1 ) n$ for $n\to infty$.
\end{thm}

\noindent\textbf{Remark:} This theorem was formulated as conjecture by D. Callan in the OEIS, entry \oeis{A090981}.\\
\\
Before we give the proof, let us first recall some central definitions and theorems of analytic combinatorics. Proofs and more details can be found in \cite{FS-anacomb}.

\begin{defi}
Let $R$ be a real number greater than one, and $\phi$ be an angle such that $0<\phi<\frac{\pi}{2}$. An open $\Delta$-\em{}domain\em{} (at $1$), denoted $\Delta (\phi, R)$ is then defined as
$$\Delta(\phi, R):=\{z: |z|<r, z\not=1, |\arg(z-1)|<\phi\}.$$
For any complex number $\zeta\not=0$ a $\Delta$-\em{}domain at\em{} $\zeta$ is the image of a $\Delta$-domain at $1$ under the mapping $z\mapsto \zeta z$. A function is called $\Delta$-\em{}analytic\em{} if it is analytic in some $\Delta$-domain.
\end{defi}

\begin{thm}
\label{thm-fs}
Let $f(z)=(1-z)^{-\alpha}$ for $\alpha\in\mathbb{C}\setminus\mathbb{Z}_{\leq 0}$. Then
$$[z^n]f(z)=\frac{n^{\alpha-1}}{\Gamma(\alpha)}\left(1+O\left(\frac{1}{n}\right)\right),$$
where $\Gamma$ denotes the Gamma-function.
\end{thm}

\begin{thm}[Transfer theorem]
\label{transfer-thm}
Suppose that $f$ satisfies in an intersection of a neighborhood of 1 with a $\Delta$-domain the condition
$$f(z)=O\left((1-z)^{-\alpha}\left(\log\frac{1}{1-z}\right)^{\beta}\right).$$
Then
$$[z^n]f(z)=O(n^{\alpha -1}(\log n)^\beta).$$
The same statement also holds for $o$-notation.
\end{thm}

\begin{cor}
\label{cor-fs}
Let $f(z)$ be $\Delta$-analytic and $f(z)\sim (1-z)^{-\alpha}$ for $z\to 1, z\in \Delta$ and $\alpha\not\in\mathbb{Z}_{\leq 0}$. Then
$$[z^n]f(z)\sim\frac{n^{\alpha -1}}{\Gamma(\alpha)}.$$
\end{cor}

\noindent
\textit{Proof of Theorem \ref{schroeder-asc}.}
The (contiguous) patterns $UD$ and $UF$ mark the end of an ascent. Thus, when counting ascents we want to count how many times these two patterns occur. Problems like this can also be dealt with the vectorial kernel method: Instead of not allowing a transition from one state to another which would complete the pattern, we mark such transitions with a new variable and then read off the corresponding coefficients in the generating function in order to obtain the number of walks where this pattern occurs $k$ times, since it is encoded by the $k$-th power of this new variable.

Our problem can be encoded by the following automaton:\\
\begin{center}
\includegraphics[height=3cm]{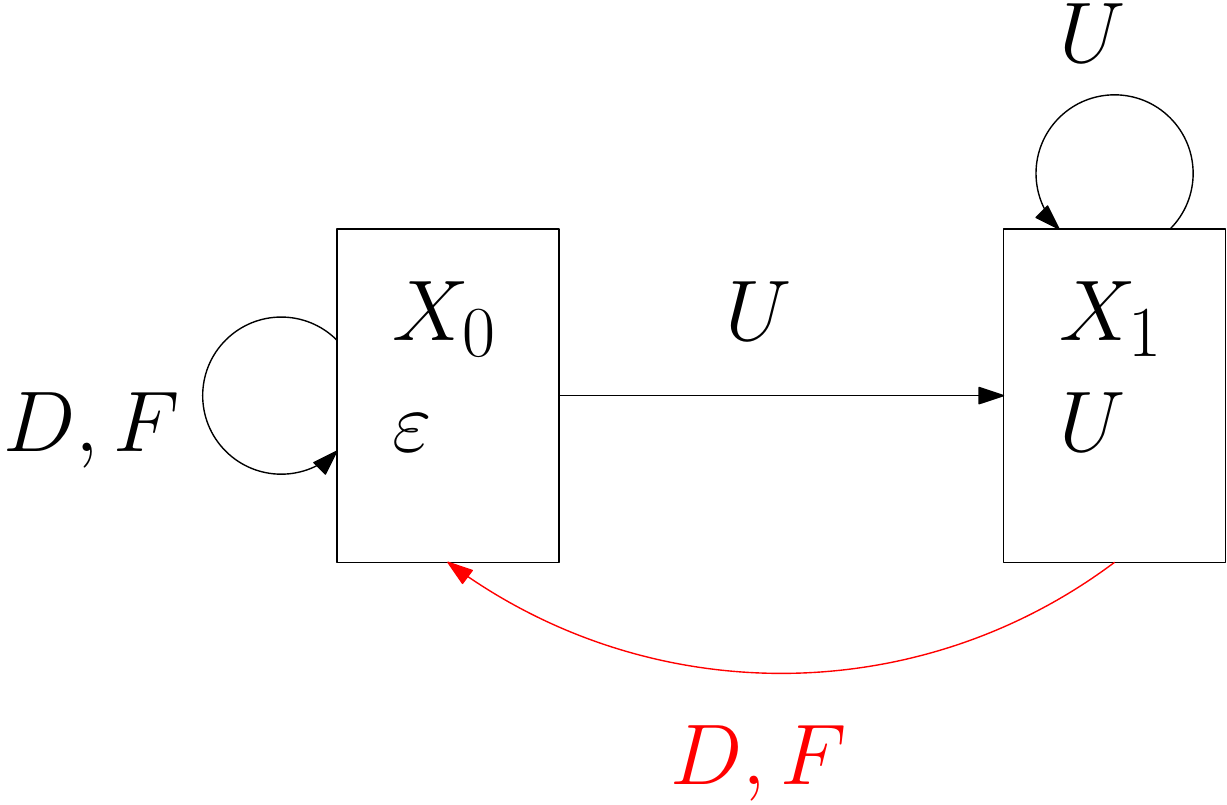}
\end{center}
The red arrow marks the ascents we want to count and will be marked by a new variable $v$ in the adjacency matrix. Its adjacency matrix is given by
$$
A=\begin{pmatrix}
tu^{-1}+t^2 & tu\\
(tu^{-1}+t^2){\red v} & tu\\
\end{pmatrix}$$
where $u$ encodes altitude, $t$ encodes length of the path, and $v$ counts the number of ascents. Thus we have that
$$I-A=\begin{pmatrix}
1-tu^{-1}-t^2 & -tu\\
-tu^{-1}v-t^2v & 1-tu\\
\end{pmatrix}.$$
The kernel is then given by
\begin{equation}
\label{kern-asc}
K(t,u)=\det(I-A)=u^{-1}((t^3-t^3v-t)u^2+(1-t^2v)u-t).
\end{equation}
Its zeroes are
$$u_{1,2}=\frac{1-t^2v\pm\sqrt{t^4(v-2)^2-2t^2(v+2)+1}}{2t(1+t^2(v-1))},$$
the one with minus being the small root. Hence, the number of small roots is $e=1$.

Writing $M_0=M_0(t,u,v)$ for the walks ending in state $X_0$ (i.e. in an $F$- or $D$-step) and $M_1=M_1(t,u)$ for the walks ending in state $X_1$ (i.e. in an $U$-step), we obtain the following vectorial functional equation
\begin{equation}
\label{feqn-asc}
(M_0,M_1)(I-A)=(1,0)-\{u^{<0}\}((M_0,M_1)A).
\end{equation}
We are interested in $M(t,0,v)=M_0(t,0,v)$, i.e. walks ending at altitude zero (since walks ending in state $X_1$ end in an up-step, they have final altitude at least 1, they will not contribute). In order to compute the forbidden vector $F=(1,0)-\{u^{<0}\}((M_0,M_1)A$ we compute
$$\{u^{<0}\}((M_0,M_1)A=(tu^{-1}M_0+t^2M_0+tu^{-1}vM_1+t^2vM_1, tu(M_0+M_1)).$$
Writing $m_0:=[u^0]M_0(t,u)$ and using $[u^0]M_1(t,u)=0$ we obtain
$$\{u^{<0}\}((M_0,M_1)A)=(tu^{-1}m_0,0)$$
and
$$F=(1-tu^{-1}m_0,0).$$

The adjoint of the adjacency matrix is given by
$$\text{adj}(I-A)=\begin{pmatrix}
1-tu & tu\\
t^2v+tu^{-1}v & -t^2-tu^{-1}+1\\
\end{pmatrix}$$
and thus the autocorrelation vector $\vec{v}$ is
$$\vec{v}=\text{adj}(I-A)\cdot\begin{pmatrix} 1\\ 1\end{pmatrix}= \begin{pmatrix} 1\\ t^2v+tu^{-1}v-t^2-tu^{-1}+1\end{pmatrix}.$$
We obtain that
$$\Phi(t,u)=u^{e}F\cdot\vec{v}=u-tm_0.$$
Using
 $$\Phi(t,u)=G(t,u)(u-u_1)$$
where $u_1$ is the small root of the kernel we obtain that $\deg_u G=0$ and by comparing coefficients we obtain that
$$G=1 \quad\text{and}\quad Gu_1=tm_0.$$
Thus we have
$$M(t,0,v)=E(t,v)=m_0=\frac{Gu_1}{t}=\frac{1-t^2v-\sqrt{t^4(v-2)^2-2t^2(v+2)+1}}{2t^2(1+t^2(v-1))}.$$
Transitioning to semilength $x:=t^2$ (and omitting the dependency on $u$) we obtain
$$E(x,v)=\frac{1-xv-\sqrt{1-2x(v+2)+x^2(v-2)^2}}{2x(1+x(v-1))}.$$
We are interested in the asymptotic behavior of
$$\mathbb{E}X_n=\frac{[x^n]\partial_v E(x,v)|_{v=1}}{[x^n]E(x,1)}.$$
We have
\begin{equation}
\label{mx1}
E(x,1)=\frac{1-x-\sqrt{1-6x+x^2}}{2x},
\end{equation}
which is the generating function of Schröder paths, and
\begin{equation}
\label{dvmx1}
\partial_v E(x,v)|_{v=1}=\frac{x^2-5x+2+(x+2)\sqrt{1-6x+x^2}}{2\sqrt{1-6x+x^2}}=\frac{x+2}{2}+\frac{x^2-5x+2}{2\sqrt{1-6x+x^2}}.
\end{equation}
By the rules for computing limits we have
$$\lim_{n\to\infty}\mathbb{E}X_n=\lim_{n\to\infty}\frac{[x^n]\partial_v E(x,v)|_{v=1}}{[x^n]E(x,1)}=\frac{\lim_{n\to\infty}[x^n]\partial_v E(x,v)|_{v=1}}{\lim_{n\to\infty}[x^n]E(x,1)}$$
thus it remains to compute the coefficient asymptotics for (\ref{mx1}) and (\ref{dvmx1}). 

First we are going to determine
$$[x^n]E(x,1)=[x^{n+1}]\frac{-\sqrt{1-6x+x^2}}{2}$$
for $n$ large. The discriminant $1-6x+x^2$ has the roots $x_{1,2}=3\pm\sqrt{8}$, where $\rho=3-\sqrt{8}$ is the dominant singularity and $3+\sqrt{8}$ lies outside every $\Delta$-domain around $\rho$. First, we want to move the dominant singularity to one in order to use the above theorems. This can be done via the substitution $z=\frac{x}{3-\sqrt{8}}$. We have that
\begin{align*}
\sqrt{1-6x+x^2} &=\sqrt{3-\sqrt{8}-x}\cdot \sqrt{3+\sqrt{8}-x}=\sqrt{3-\sqrt{8}}\sqrt{1-z}\cdot\sqrt{3+\sqrt{8}-(3-\sqrt{8})z}\\
			&\sim(3-\sqrt{8})^{1/2}(2\sqrt{8})^{1/2}\sqrt{1-z}
\end{align*}
locally for $z\to 1$. Thus, by Corollary \ref{cor-fs} with $\alpha=-\frac{1}{2}$ we have that
\begin{align}
[x^n]E(x,1)	&\sim [x^{n+1}]\frac{1}{2}(2\sqrt{8}(3-\sqrt{8}))^{1/2}\left(-\sqrt{1-\frac{x}{3-\sqrt{8}}}\right)\nonumber\\
			&= -\frac{1}{2} (2\sqrt{8}(3-\sqrt{8}))^{1/2} (3-\sqrt{8})^{-n-1}[z^{n+1}]\sqrt{1-z}\nonumber\\
			&= -\frac{1}{2} (2\sqrt{8})^{1/2} (3-\sqrt{8})^{-n-1/2}\frac{(n+1)^{-3/2}}{\Gamma(\frac{1}{2})}\nonumber\\
			&\sim \frac{1}{2} (3-\sqrt{8})^{-n-1/2} \frac{(2\sqrt{8})^{1/2}}{2\sqrt{\pi}}n^{-3/2}\label{xnm}
\end{align}
for $n\to\infty$.
In order to compute $[x^n]\partial_vE(x,v)|_{v=1}$ we first compute $[x^n](1-6x+x^2)^{-1/2}$ because this expression will appear in the computation of  $[x^n]\partial_vE(x,v)|_{v=1}$. By the substitution $z=\frac{x}{3-\sqrt{8}}$ and Corollary \ref{cor-fs} with $\alpha=\frac{1}{2}$ we obtain
\begin{align}
[x^n](1-6x+x^2)^{-1/2}	&= [x^n]((3-\sqrt{8})-x)^{-1/2}((3+\sqrt{8})-x)^{-1/2}\nonumber\\
					&= [z^n](3-\sqrt{8})^{-n-1/2}(1-z)^{-1/2}((3+\sqrt{8})-(3-\sqrt{8})z)^{-1/2}\nonumber\\
					&\sim (3-\sqrt{8})^{-n-1/2}(2\sqrt{8})^{-1/2}\frac{n^{-1/2}}{\sqrt{\pi}} \label{aux}
\end{align}
for $n\to \infty$.
For $n$ large we have that
\begin{align*}
[x^n]\partial_vE(x,v)|_{v=1} 	&=\frac{1}{2}[x^n](x^2-5x+2)(1-6x+x^2)^{-1/2}\\
						&= \frac{1}{2}[x^{n-2}](1-6x+x^2)^{-1/2}-\frac{5}{2}[x^{n-1}](1-6x+x^2)^{-1/2}+[x^n](1-6x+x^2)^{-1/2}.
\end{align*}
Using (\ref{aux}) and the fact that $(n-k)^{-1/2}\sim n^{-1/2}$ for $k$ constant and $n\to \infty$ we obtain after some simplifications that
\begin{equation}
\label{xndvmx1}
[x^n]\partial_vE(x,v)|_{v=1} \sim \frac{(2\sqrt{8})^{1/2}}{\sqrt{\pi}}n^{-1/2}(3-\sqrt{8})^{-n-1/2}(2-\sqrt{2})
\end{equation}
Using the expressions for (\ref{xnm}) and (\ref{xndvmx1}) we obtain that for $n\to\infty$ the expected value of ascents behaves like
$$\mathbb{E}X_n\sim\frac{(3-\sqrt{8})^{-n-1/2}(2-\sqrt{2})}{\sqrt{\pi}n^{1/2}(2\sqrt{8})^{1/2}}\cdot\frac{2\cdot2\sqrt{\pi}n^{3/2}}{(3-\sqrt{8})^{-n-1/2}(2\sqrt{8})^{1/2}}$$
which, after some simplifications becomes
\begin{equation}
\label{mu}
\mathbb{E}X_n\sim(\sqrt{2}-1)n.
\end{equation}
This proves Callans conjecture. \hfill $\Box$


\begin{thm}
Let $X_n$ be the random variable counting ascents in a Schröder path of length $n$ which is chosen uniformly at random among all Schröder paths of length $n$. Then
\begin{equation}
\label{sigma}
\mathbb{V}X_n\sim\frac{188-133\sqrt{2}}{8\sqrt{2}-12}n\approx 0.1317\, n
\end{equation}
for $n\to \infty$.
\end{thm}
\begin{proof}
The variance can be computed using similar means as the expected value. We have that
\begin{equation}
\label{var}
\mathbb{V}(X_n) =\frac{[x^n]\partial_v^2E(x,v)|_{v=1}}{[x^n]E(x,1)}+\frac{[x^n]\partial_vE(x,v)|_{v=1}}{[x^n]A(x,1)}-\left(\frac{[x^n]\partial_vE(x,v)|_{v=1}}{[x^n]A(x,1)}\right)^2.
\end{equation}
The second derivative of $E$ with respect to $v$ is given by
$$\partial_v^2E(x,v)|_{v=1}=(-x^5+11x^4-33x^2+21x^2+2x)(x^2-6x+1)^{-3/2}-\frac{x^4-8x^3+13x^2-2x}{x^2-6x+1}.$$

Using the substitution $z=\frac{x}{3-\sqrt{8}}$ and the tables for the asymptotics of standard functions from \cite{FS-anacomb}, p. 388 we see that
\begin{align*}
[z^n](1-z)^{1/2} &\sim -\frac{1}{\sqrt{\pi n^3}}\left(\frac{1}{2}+\frac{3}{16n}+\frac{25}{256n^2}+\mathcal{O}\left(\frac{1}{n^3}\right)\right),\\
[z^n](1-z)^{-1/2} &\sim \frac{1}{\sqrt{\pi n}}\left(1-\frac{1}{8n}+\frac{1}{128n^2}+\mathcal{O}\left(\frac{1}{n^3}\right)\right),\\
[z^n](1-n)^{-1} &\sim 1\\
[z^n](1-z) &\sim \sqrt{\frac{n}{\pi}}\left(2+\frac{3}{4n}-\frac{7}{64n^2}\mathcal{O}\left(\frac{1}{n^3}\right)\right)
\end{align*}
(we need the additional terms because there will be a cancellation of the leading terms of order $n^2$, just the previously computed terms will not do the trick).

Plugging these as well as the correct asymptotic growth rates in the formula for the variance (\ref{var}) we obtain the claim of the theorem after some cancellations and computing limits.
\end{proof}

With the help of the Drmota-Lalley-Woods theorem we can obtain even more information about the limiting distribution of the number of ascents.

\begin{thm}[Drmota-Lalley-Woods theorem, limiting distribution version from \cite{BD}] Suppose that $\mathbf{y}=\mathbf{P}(z,\mathbf{y},u)$ is a strongly connected and analytically well defined entire or polynomial system of equations that depends on $u$ and has a solution $\mathbf{f}$ that exists in a neighborhood of $u=1$. Furthermore, let $h(z,u)$ be given by
$$h(z,u)=\sum_{n\geq 0} h_n(u)z^n = H(z,\mathbf{f}(z,u), u),$$
where $H(z,y,u)$ is entire or a polynomial function with non-negative coefficients that depends on $\mathbf{y}$ and suppose that $h_n(u)\not=0$ for all $n\geq n_0$ (for some $n_0\geq 0$).

Let $X_n$ be a random variable whose distribution is defined by
$$\mathbb{E}[\left[u^{X_n}\right]=\frac{h_n(u)}{h_n(1)}.$$
Then $X_n$ has a Gaussian limiting distribution. More precisely, we have $\mathbb{E}[X_n]=\mu n+O(1)$ and $\mathbb{V}[X_n]=\sigma^2 n+O(1)$ for constants $\mu>0$ and $\sigma^2\geq 0$ and 
$$\frac{1}{\sqrt{n}}(X_n-\mathbb{E}[X_n])\to N(0,\sigma^2).$$
\end{thm}
\noindent
\begin{proof}
See \cite{BD} or \cite{Drmota}.
\end{proof}

\begin{cor}
The number of ascents in Schröder paths has a Gaussian limiting distribution with parameters $\mu=\sqrt{2}-1$ and $\sigma^2=\frac{188-133\sqrt{2}}{8\sqrt{2}-12}$.
\end{cor}
\noindent
\textit{Proof.} Let 
$$P(z,y,u)=z(1+z(u-1))y^2+zuy+1.$$
Solving the system $y=P(z,y,u)$ gives us
$$f(z,u)=\frac{1-zu-\sqrt{1-2z(u+2)+z^2(u-2)^2}}{2z(1-z(u-1))}$$
which is a formal power series in a neighborhood of $u=1$ (the other solution with plus is not and can be disregarded). The function $f$ coincides with $E(x,v)$ (after a substitution $z=x$ and $u=v$). The system is strongly connected since it consists of only one equation in one unknown. Let $H(z,y,u)=y$ such that $H(z,f,u)=f(z,u)$. From the combinatorial interpretation we see that $h_n(u)\not=0$ for $n\geq n_0$ (remember, $h_n(u)$ counts ascents in Schröder paths of length $n$, thus being a power series of the form $1+c_1u+O(u^2)$ for any $n>0$, the 1 comes from the Schröder path consisting only of flat steps, thus having no ascent). The random variable $X_n$ counting ascents has distribution defined by
$$\mathbb{E}\left[u^{X_n}\right]=\frac{h_n(u)}{h_n(1)}.$$
Thus, we can apply the Drmota-Lalley-Woods theorem and obtain that $X_n$ has Gaussian limiting distribution. We already computed the constants $\mu=\sqrt{2}-1$ and $\sigma^2=\frac{188-133\sqrt{2}}{8\sqrt{2}-12}$ earlier in Equations \ref{mu} and \ref{sigma}. \hspace{\fill}$\Box$

\section{Conclusion}

The vectorial kernel method is a powerful tool, unifying various results on the enumeration of lattice paths which avoid a given pattern. In this paper the vectorial kernel method was generalized to lattice paths with longer steps and used to prove a conjecture on the asymptotic behavior of the expected number of ascents in Schröder paths. The results from this paper also allow to tackle other parameters (e.g. humps, peaks or plateaus) of paths with longer steps obeying some constraints that can be described by a finite automaton which might become a subject of further studies.

\end{document}